\documentclass[11pt]{article}
\usepackage{geometry,amsmath,amsthm,amssymb,latexsym}

\usepackage[svgnames]{xcolor}
\usepackage[colorlinks=true,linkcolor=MidnightBlue,citecolor=MidnightBlue,urlcolor=MidnightBlue,pdfpagelabels=false]{hyperref}

\usepackage{thmtools}
\theoremstyle{plain}
  \declaretheorem[numberwithin=section]{theorem}

  \declaretheorem[numberlike=theorem]{lemma}
  \declaretheorem[numberlike=theorem]{conjecture}
  
\theoremstyle{definition}
  
  \declaretheorem[numberlike=theorem]{example}
  \declaretheorem[numberlike=theorem]{remark}
\newenvironment{acknowledgements}{\bigskip\textbf{Acknowledgements.}}{}

\newcommand{\assign}{:=}
\newcommand{\md}{\mathrm{d}}
\newcommand{\nequiv}{\mathrel{\not\equiv}}
\renewcommand{\Re}{\operatorname{Re}\,}

\begin{document}

\title{Gessel--Lucas congruences for sporadic sequences}

\author{Armin Straub\thanks{\textit{Email:} \texttt{straub@southalabama.edu}}\\
Department of Mathematics and Statistics\\
University of South Alabama}

\date{January 28, 2023}

\maketitle

\begin{abstract}
  For each of the $15$ known sporadic Ap\'ery-like sequences, we prove
  congruences modulo $p^2$ that are natural extensions of the Lucas
  congruences modulo $p$. This extends a result of Gessel for the numbers used
  by Ap\'ery in his proof of the irrationality of $\zeta (3)$. Moreover, we
  show that each of these sequences satisfies two-term supercongruences modulo
  $p^{2 r}$. Using special constant term representations recently discovered
  by Gorodetsky, we prove these supercongruences in the two cases that
  remained previously open.
\end{abstract}

\section{Introduction}

Sequences $A (n)$ that are integer solutions of either the three-term
recurrence
\begin{equation}
  (n + 1)^2 A (n + 1) = (a n^2 + a n + b) A (n) - c n^2 A (n - 1),
  \label{eq:apery2:rec:abc}
\end{equation}
or the three-term recurrence
\begin{equation}
  (n + 1)^3 A (n + 1) = (2 n + 1) (a n^2 + a n + b) A (n) - n (c n^2 + d) A (n
  - 1), \label{eq:apery3:rec:abcd}
\end{equation}
subject to the initial conditions $A (- 1) = 0$, $A (0) = 1$, are known as
\emph{Ap\'ery-like sequences}. For instance, if $(a, b, c, d) = (17, 5, 1,
0)$, then the solution to \eqref{eq:apery3:rec:abcd} is the integer sequence
\begin{equation}
  A (n) = \sum_{k = 0}^n \binom{n}{k}^2 \binom{n + k}{k}^2, \label{eq:apery3}
\end{equation}
which formed the basis of Ap\'ery's proof \cite{apery}, \cite{alf} of
the irrationality of $\zeta (3)$. Systematic searches for Ap\'ery-like
sequences have been conducted by Zagier \cite{zagier4} in the case of
\eqref{eq:apery2:rec:abc}, and by Almkvist--Zudilin \cite{az-de06} and
Cooper \cite{cooper-sporadic} in the case of \eqref{eq:apery3:rec:abcd}.
After normalizing, and apart from degenerate cases as well four hypergeometric
and four Legendrian solutions in each case, only a small number of
\emph{sporadic} solutions have been found. Namely, Zagier \cite{zagier4}
found six sporadic solutions to \eqref{eq:apery2:rec:abc}, labeled
$\boldsymbol{A}, \boldsymbol{B}, \boldsymbol{C}, \boldsymbol{D}, \boldsymbol{E},
\boldsymbol{F}$. Almkvist--Zudilin \cite{az-de06} found six corresponding
sporadic solutions to \eqref{eq:apery2:rec:abc}, labeled $(\alpha), (\gamma),
(\delta), (\epsilon), (\eta), (\zeta)$, and Cooper \cite{cooper-sporadic}
found three additional sporadic solutions to \eqref{eq:apery2:rec:abc} with $d
\neq 0$, labeled $s_7, s_{10}, s_{18}$. Explicit formulas for the sequences
$(\gamma)$, $\boldsymbol{F}$, $(\delta)$, $\boldsymbol{D}$ can be found in
\eqref{eq:apery3}, \eqref{eq:AF}, \eqref{eq:Ad}, \eqref{eq:AD}, respectively.
Tables of all $6 + 6 + 3 = 15$ sequences, including known representations as
binomial sums such as \eqref{eq:apery3} can be found, for instance, in
\cite{ms-lucascongruences}, \cite{oss-sporadic} or \cite{gorodetsky-ct}.

One reason that Ap\'ery-like sequences have received attention in the
literature is that they share (or, are believed to share) various remarkable
arithmetic properties. For instance, they are connected to modular forms in
different ways \cite{sb-picardfuchs}, \cite{beukers-apery87},
\cite{ahlgren-ono-apery}, \cite{os-zagierapery} and they satisfy unusually
strong congruences \cite{beukers-apery85}, \cite{coster-sc},
\cite{ccs-apery}, \cite{os-supercong11}, \cite{os-domb},
\cite{s-apery}, \cite{oss-sporadic}, \cite{at-az},
\cite{gorodetsky-cong-q}, \cite{gorodetsky-ct} that were coined
\emph{supercongruences} by Beukers. As we review in
Section~\ref{sec:supercongruences}, these have been proven by various authors
for $13$ of the $15$ sequences. We prove the supercongruences for one of the
previously conjectural cases, namely the sequence labeled $\boldsymbol{F}$. For
the other missing sequence, labeled $(\delta)$ and known as the
Almkvist--Zudilin numbers, we prove a weaker version of the congruences which,
however, is sufficient for our present purposes. Combined with the previously
known cases, this results in the following uniform result.

\begin{theorem}
  \label{thm:supercongruences:p2}Let $A (n)$ be one of the $6 + 6 + 3$ known
  sporadic Ap\'ery-like sequences. Then, for all primes $p \geq 3$ and
  all positive integers $n, r$,
  \begin{equation*}
    A (p^r n) \equiv A (p^{r - 1} n) \pmod{p^{2 r}} .
  \end{equation*}
\end{theorem}

As indicated above, Theorem~\ref{thm:supercongruences:p2} was previously known
except in the two cases $\boldsymbol{F}$ and $(\delta)$. We prove these two new
cases as Theorems~\ref{thm:supercongruences:p2:F} and
\ref{thm:supercongruences:p2:d} in Section~\ref{sec:supercongruences}. We note
that the restriction to primes $p \geq 5$, or $p \geq 3$ in the case
of Theorem~\ref{thm:supercongruences:p2}, is natural for supercongruences due
to Lemma~\ref{lem:jacobsthal}. Numerical evidence suggests that
Theorem~\ref{thm:supercongruences:p2} also holds for $p = 2$ except for the
three sequences labeled $\boldsymbol{B}$, $(\delta)$ and $(\eta)$. We will not
attempt to discuss the case $p = 2$ in further detail here. We note that, in
the literature, the case $p = 3$ of Theorem~\ref{thm:supercongruences:p2} in
the previously known cases is not always discussed in detail. However, the
arguments, which for primes $p \geq 5$ yield slightly stronger
congruences due to Lemma~\ref{lem:jacobsthal}, still apply for $p = 3$ and the
resulting congruences are sufficient for
Theorem~\ref{thm:supercongruences:p2}.

As another instance of a special arithmetic property, Gessel showed
\cite[Theorem 1]{gessel-super} that the Ap\'ery numbers \eqref{eq:apery3}
satisfy the congruences
\begin{equation}
  A (n) \equiv A (n_0) A (n_1) \cdots A (n_r) \pmod{p},
  \label{eq:lucas:i}
\end{equation}
where $n = n_0 + n_1 p + \cdots + n_r p^r$ is the $p$-adic expansion of $n$.
The congruences \eqref{eq:lucas:i} are known as \emph{Lucas congruences}
because they are of the same kind as the congruences that Lucas
\cite{lucas78} showed for the binomial coefficients. Various sequences and
families of sequences have since been shown \cite{mcintosh-lucas},
\cite{granville-bin97}, \cite{sd-laurent09}, \cite{ry-diag13},
\cite{ms-lucascongruences}, \cite{delaygue-apery}, \cite{abd-lucas},
\cite{gorodetsky-ct} to satisfy the Lucas congruences. For a historical
survey, we further refer to \cite{mestrovic-lucas}.

In particular, it turns out that, as shown by Malik and the author
\cite{ms-lucascongruences} (see also \cite{gorodetsky-ct}), each sporadic
Ap\'ery-like sequence $A (n)$ satisfies the Lucas congruences.

\begin{theorem}[\cite{ms-lucascongruences}]
  \label{thm:lucas}Let $A (n)$ be one of the $6 + 6 + 3$ known sporadic
  Ap\'ery-like sequences. Then, for all primes $p$ and all integers $n, k$
  with $0 \leq k < n$,
  \begin{equation}
    A (p n + k) \equiv A (k) A (n) \pmod{p} . \label{eq:lucas:p:i}
  \end{equation}
\end{theorem}

By iterating, one sees that the congruences \eqref{eq:lucas:p:i} are
equivalent to the the Lucas congruences \eqref{eq:lucas:i}. Gessel
\cite{gessel-super} further proved an extension of the Lucas congruences for
the Ap\'ery numbers \eqref{eq:apery3} modulo $p^2$. The second main result
of this paper is to prove corresponding congruences for all known sporadic
Ap\'ery-like sequences. This results in the following extension of
Theorem~\ref{thm:lucas} modulo $p^2$.

\begin{theorem}
  \label{thm:lucas:p2:i}Let $A (n)$ be one of the $6 + 6 + 3$ known sporadic
  Ap\'ery-like sequences. Then, for all primes $p \geq 3$ and all
  integers $n, k$ with $0 \leq k < n$,
  \begin{equation}
    A (p n + k) \equiv A (k) A (n) + p n A' (k) A (n) \pmod{p^2} .
    \label{eq:lucas:p2:i}
  \end{equation}
\end{theorem}

We will refer to the congruences \eqref{eq:lucas:p2:i} as \emph{Gessel--Lucas
congruences} modulo $p^2$. For a precise definition of the formal derivative
$A' (n)$, we refer to Section~\ref{sec:A:prime}. However, for certain of the
sequences, $A' (n)$ can be obtained as the actual derivative
$\frac{\md}{\md n} A (n)$ of a natural interpolation of $A (n)$. We
prove Theorem~\ref{thm:lucas:p2:i} in Section~\ref{sec:proof} by extending the
proof of Gessel \cite{gessel-super}, who proved the congruences
\eqref{eq:lucas:p2:i} for the Ap\'ery numbers \eqref{eq:apery3}. To our
knowledge, the Gessel--Lucas congruences \eqref{eq:lucas:p2:i} had not been
previously observed for the other sequences covered by
Theorem~\ref{thm:lucas:p2:i}. A crucial ingredient for the proof is
Theorem~\ref{thm:supercongruences:p2}, the supercongruences satisfied by each
sporadic sequence.

The reductions modulo prime powers $p^r$ of certain sequences $A (n)$, such as
diagonals of rational functions and constant terms, can be described using the
notion of \emph{linear $p$-schemes} that was recently introduced by Rowland
and Zeilberger \cite{rz-cong}. Using the language found in
\cite{beukers-p-schemes}, a linear $p$-scheme modulo $p^r$ for $A (n)$ with
$s$ states consists of a vector of sequences $\boldsymbol{A} (n) = (A_1 (n),
\ldots, A_s (n))$, with $A_1 (n) = A (n)$, as well as matrices $M (k)$ for $k
\in \{ 0, 1, \ldots, p - 1 \}$ such that
\begin{equation*}
  \boldsymbol{A} (p n + k) \equiv M (k) \boldsymbol{A} (n) \pmod{p^r}
\end{equation*}
for all integers $n, k$ with $0 \leq k < n$. Based on earlier work of
Rowland and Yassawi \cite{ry-diag13}, Rowland and Zeilberger describe in
\cite{rz-cong} algorithms to compute linear $p$-schemes for the values
modulo $p^r$ of a sequence $A (n) = \operatorname{ct} [P (\boldsymbol{x})^n Q
(\boldsymbol{x})]$ of constant terms, where $P, Q \in \mathbb{Z}
[\boldsymbol{x}^{\pm 1}]$ are Laurent polynomials in $\boldsymbol{x}= (x_1,
\ldots, x_d)$. As done in \cite{ry-diag13} and \cite{s-schemes}, one can
derive upper bounds for the number of states for the $p$-schemes produced by
these algorithms. However, as pointed out by Beukers
\cite{beukers-p-schemes}, these bounds appear inefficient except in small
cases.

On the other hand, it appears to be an interesting and fruitful question to
ask for general results to describe sequences whose reductions modulo $p^r$
can be expressed using linear $p$-schemes with few states. For instance, as
observed in \cite{hs-lucas-x}, a sequence $A (n)$ satisfies the Lucas
congruences \eqref{eq:lucas:i} modulo $p$ if and only if its modulo $p$
reductions can be encoded by a linear $p$-scheme with a single state. One way
to interpret Theorem~\ref{thm:lucas:p2:i} is that it provides explicit
two-state linear $p$-schemes for all sporadic Ap\'ery-like sequences modulo
$p^2$. It would be of considerable interest to better understand which other
sequences share this property generalizing the classical Lucas congruences.

\section{Notation and basic congruences}

Suppose that $\boldsymbol{k}= (k_1, \ldots, k_{\ell})$ is a tuple of nonnegative
integers adding up to $n$, that is, $| \boldsymbol{k} | = k_1 + k_2 + \cdots +
k_{\ell} = n$. Then we denote the corresponding multinomial coefficient as
\begin{equation*}
  \binom{n}{\boldsymbol{k}} = \binom{n}{k_1, \ldots, k_{\ell}} = \frac{n!}{k_1
   !k_2 ! \cdots k_{\ell} !} .
\end{equation*}
Throughout the paper, we use typical notation and write, for instance,
$\lambda \boldsymbol{k}$ as short for $(\lambda k_1, \lambda k_2, \ldots,
\lambda k_{\ell})$.

The following version of Jacobsthal's binomial congruence is proved in
\cite{gessel-euler} and \cite{granville-bin97} (and is extended in
\cite{s-apery} to binomial coefficients which are allowed to have negative
entries).

\begin{lemma}
  \label{lem:jacobsthal}For primes $p \geq 5$, and integers $n, k$ and
  $r, s \geq 1$,
  \begin{equation}
    \binom{p^r n}{p^s k} / \binom{p^{r - 1} n}{p^{s - 1} k} \equiv 1 \pmod{p^{r + s + \min (r, s)}} . \label{eq:jacobsthal}
  \end{equation}
  For primes $p = 3$, respectively $p = 2$, the congruence
  \eqref{eq:jacobsthal} holds modulo $p^{r + s + \min (r, s) - 1}$,
  respectively $p^{r + s + \min (r, s) - 2}$.
\end{lemma}

Also note that, if $p \nmid k$ and $s \leq r$, then we have the much
simpler congruence
\begin{equation}
  \binom{p^r n}{p^s k} = p^{r - s}  \frac{n}{k} \binom{p^r n - 1}{p^s k - 1}
  \equiv 0 \pmod{p^{r - s}} . \label{eq:binomial:0:rs}
\end{equation}
Finally, for the proof of our main result, we record the following observation
concerning Cooper's sporadic sequences $s_7, s_{10}, s_{18}$.

\begin{lemma}
  \label{lem:A:pm1}If $A (n)$ is one of the $3$ known sporadic Ap\'ery-like
  sequences that satisfy \eqref{eq:apery3:rec:abcd} with $d \neq 0$ then, for
  all primes $p$,
  \begin{equation*}
    A (p - 1) \equiv 0 \pmod{p} .
  \end{equation*}
\end{lemma}

\begin{proof}
  This is a special case of \cite[Theorem~6.6]{ms-lucascongruences} and the
  discussion preceding it.
\end{proof}

\section{Supercongruences}\label{sec:supercongruences}

It was observed by Beukers \cite{beukers-apery85} that the Ap\'ery numbers
satisfy congruences of a certain type that are stronger than what is predicted
by formal group theory. Such \emph{supercongruences} were further studied by
Coster \cite{coster-sc} who showed that the Ap\'ery numbers satisfy
\begin{equation*}
  A (p^r n) \equiv A (p^{r - 1} n) \pmod{p^{3 r}}
\end{equation*}
for all primes $p \geq 5$. The case $r = 1$ of these congruences had
previously been observed by Gessel \cite{gessel-super}. Since then
supercongruences have been established for other Ap\'ery-like sequences
\cite{os-supercong11}, \cite{os-domb}, \cite{oss-sporadic},
\cite{gorodetsky-cong-q}, \cite{gorodetsky-ct}. In fact, numerical
evidence suggests that all Ap\'ery-like sequences satisfy supercongruences.
More precisely, the following conjecture appears in \cite{oss-sporadic}.

\begin{conjecture}
  \label{conj:supercongruence}Let $A (n)$ be one of the $6 + 6 + 3$ known
  sporadic Ap\'ery-like sequences. Then, for all primes $p \geq 5$ and
  all positive integers $n, r$,
  \begin{equation}
    A (p^r n) \equiv A (p^{r - 1} n) \pmod{p^{\lambda r}}
    \label{eq:supercongruence}
  \end{equation}
  where $\lambda = 3$ except in the five cases $\boldsymbol{B}$, $\boldsymbol{C}$,
  $\boldsymbol{E}$, $\boldsymbol{F}$ and $s_{18}$ in which case $\lambda = 2$.
\end{conjecture}

The four cases $\boldsymbol{A}$, $\boldsymbol{D}$, $(\gamma)$, $s_{10}$ follow
from Coster's work \cite{coster-sc}. Osburn and Sahu proved the case
$\boldsymbol{C}$ in \cite{os-supercong11} as well as the cases $\boldsymbol{E}$
and $(\alpha)$ in \cite{os-domb}; together with the author, they further
established the cases $(\epsilon)$, $(\eta)$, $s_7$ and $s_{18}$ in
\cite{oss-sporadic}. More recently, Gorodetsky proved the case $(\zeta)$ in
\cite{gorodetsky-cong-q} as well as the case $\boldsymbol{B}$
\cite{gorodetsky-ct}.

In Theorem~\ref{thm:supercongruences:p2:F}, we prove the previously open case
$\boldsymbol{F}$ of Conjecture~\ref{conj:supercongruence}. For the final missing
case $(\delta)$, we prove in Theorem~\ref{thm:supercongruences:p2:d} a weaker
version of Conjecture~\ref{conj:supercongruence} where the conjectured
exponent $\lambda = 3$ is replaced with $\lambda = 2$. The crucial ingredient
for these proofs are suitable constant terms expressions that were recently
obtained by Gorodetsky \cite{gorodetsky-ct}. Combined with the previously
known cases, this proves Theorem~\ref{thm:supercongruences:p2}, which is the
weaker version of Conjecture~\ref{conj:supercongruence} where $\lambda = 2$
for all sequences. While the case $(\delta)$ of
Conjecture~\ref{conj:supercongruence} remains open in general, we note that
Amdeberhan and Tauraso \cite{at-az} prove the case $r = 1$ of the
corresponding conjectured supercongruences \eqref{eq:supercongruence}.

We recall that the sporadic Ap\'ery-like sequence labeled $\boldsymbol{F}$ by
Zagier \cite{zagier4} is the sequence
\begin{equation}
  A_{\boldsymbol{F}} (n) = \sum_{k = 0}^n (- 1)^k 8^{n - k} \binom{n}{k} \sum_{l
  = 0}^k \binom{k}{l}^3, \label{eq:AF}
\end{equation}
which solves the three-term recurrence \eqref{eq:apery2:rec:abc} with $(a, b,
c) = (17, 6, 72)$.

\begin{theorem}
  \label{thm:supercongruences:p2:F}For all primes $p \geq 3$ and all
  positive integers $n, r$,
  \begin{equation*}
    A_{\boldsymbol{F}} (p^r n) \equiv A_{\boldsymbol{F}} (p^{r - 1} n) \pmod{p^{2 r}} .
  \end{equation*}
\end{theorem}

\begin{proof}
  By expressing $A_{\boldsymbol{F}} (n)$ as the constant term of $\Lambda (x,
  y)^n$ where
  \begin{equation*}
    \Lambda (x, y) = \frac{(x - y + 1) (y - x + 1) (x + y - 1) (x + y + 1)
     (x^2 + y^2 + 1)}{(x y)^2},
  \end{equation*}
  Gorodetsky \cite{gorodetsky-ct} obtained the alternative representation
  \begin{equation*}
    A_{\boldsymbol{F}} (n) = \sum_{(\boldsymbol{a}, \boldsymbol{b}, \boldsymbol{c},
     \boldsymbol{d}, \boldsymbol{e}) \in U (n)} (- 1)^{a_1 + b_2 + c_3}
     \binom{n}{\boldsymbol{a}} \binom{n}{\boldsymbol{b}} \binom{n}{\boldsymbol{c}}
     \binom{n}{\boldsymbol{d}} \binom{n}{\boldsymbol{e}}
  \end{equation*}
  as a multiple binomial sum. Here, the sum is over the set
  \begin{equation*}\arraycolsep=0.2em
    U (n) = \left\{ (\boldsymbol{a}, \boldsymbol{b}, \boldsymbol{c}, \boldsymbol{d},
     \boldsymbol{e}) \in \mathbb{Z}_{\geq 0}^{15} \, : \quad
     \begin{array}{rcc}
       a_1 + a_2 + a_3 & = & n\\
       b_1 + b_2 + b_3 & = & n\\
       c_1 + c_2 + c_3 & = & n\\
       d_1 + d_2 + d_3 & = & n\\
       e_1 + e_2 + e_3 & = & n
     \end{array}, \quad \begin{array}{rcc}
       a_i + b_i + c_i + d_i + 2 e_i & = & 2 n\\
       \text{for each $i \in \{ 1, 2, 3 \}$} &  & 
     \end{array} \right\} .
  \end{equation*}
  In the sequel, we use the notation $\boldsymbol{k}= (\boldsymbol{a},
  \boldsymbol{b}, \boldsymbol{c}, \boldsymbol{d}, \boldsymbol{e}) \in
  \mathbb{Z}_{\geq 0}^{15}$. As usual, we say that $p$ divides
  $\boldsymbol{k}$ (and write $p|\boldsymbol{k}$) if $p$ divides each component of
  $\boldsymbol{k}$. If $\boldsymbol{k} \in U (n)$, then we write
  \begin{equation*}
    B (\boldsymbol{k}) \assign (- 1)^{a_1 + b_2 + c_3} \binom{n}{\boldsymbol{a}}
     \binom{n}{\boldsymbol{b}} \binom{n}{\boldsymbol{c}} \binom{n}{\boldsymbol{d}}
     \binom{n}{\boldsymbol{e}} .
  \end{equation*}
  Using this notation, we have
  \begin{equation}
    A_{\boldsymbol{F}} (p^r n) = \sum_{\boldsymbol{k} \in U (p^r n)} B
    (\boldsymbol{k}) = \sum_{\substack{
      \boldsymbol{k} \in U (p^r n)\\
      p|\boldsymbol{k}
    }} B (\boldsymbol{k}) + \sum_{\substack{
      \boldsymbol{k} \in U (p^r n)\\
      p \nmid \boldsymbol{k}
    }} B (\boldsymbol{k}) . \label{eq:AF:B12}
  \end{equation}
  We claim (and will show below) that, for all primes $p \geq 3$ and all
  $\boldsymbol{k} \in U (p^r n)$, we have
  \begin{equation}
    B (\boldsymbol{k}) \equiv B (\boldsymbol{k}/ p) \quad \left({\operatorname{mod} p^{2
    r}}  \right) \label{eq:summand:F:p2}
  \end{equation}
  where the right-hand side is to be interpreted as $0$ if $p \nmid
  \boldsymbol{k}$. Combining \eqref{eq:AF:B12} and \eqref{eq:summand:F:p2}, we
  conclude that
  \begin{equation*}
    A_{\boldsymbol{F}} (p^r n) \equiv \sum_{\substack{
       \boldsymbol{k} \in U (p^r n)\\
       p|\boldsymbol{k}
     }} B (\boldsymbol{k}/ p) = \sum_{\substack{
       \boldsymbol{k} \in U (p^{r - 1} n)
     }} B (\boldsymbol{k}) = A_{\boldsymbol{F}} (p^{r - 1} n) \pmod{p^2},
  \end{equation*}
  which is what we set out to prove. Note that the middle equality uses that
  $\boldsymbol{k} \in U (p^{r - 1} n)$ if and only if $p\boldsymbol{k} \in U (p^r
  n)$.
  
  It therefore only remains to prove \eqref{eq:summand:F:p2}. First, consider
  the case $\boldsymbol{k} \in U (p^r n)$ with $p \nmid \boldsymbol{k}$. Without
  loss of generality, we may assume that $p \nmid a_i$ for some $i \in \{ 1,
  2, 3 \}$ (otherwise, the same argument applies with $a_i$ replaced by one of
  $b_i, c_i, d_i, e_i$). It follows from the constraint $a_i + b_i + c_i + d_i
  + 2 e_i = 2 p^r n$ that at least one of $b_i, c_i, d_i, e_i$ is also not
  divisible by $p$ (note that we excluded the case $p = 2$ so that the factors
  of $2$ in the constraint can be ignored). Without loss, we assume that $p
  \nmid b_i$. Note that we have the divisibility
  \begin{equation}
    \binom{p^r n}{a_i} \binom{p^r n}{b_i} \quad | \quad B (\boldsymbol{k}),
    \label{eq:summand:F:div}
  \end{equation}
  and that, by \eqref{eq:binomial:0:rs}, the two binomial coefficients on the
  left-hand side are each divisible by $p^r$. It follows that $B
  (\boldsymbol{k})$ is divisible by $p^{2 r}$, which proves congruence
  \eqref{eq:summand:F:p2} in the case $p \nmid \boldsymbol{k}$.
  
  Finally, consider \eqref{eq:summand:F:p2} in the case $p|\boldsymbol{k}$.
  Write $s = \min (\nu_p (a_1), r)$ and $t = \min (\nu_p (a_2), r)$ and
  suppose that $s \geq t$ (this is no loss of generality; otherwise, we
  can swap $a_1$ and $a_2$ throughout the proof). Consequently, $a_3 = p^r n -
  a_1 - a_2$ is divisible by $p^t$ as well. It follows from
  Lemma~\ref{lem:jacobsthal} applied to each binomial coefficient on the
  right-hand side of
  \begin{equation*}
    \binom{p^r n}{\boldsymbol{a}} = \binom{p^r n}{a_1} \binom{p^r n - a_1}{a_2}
  \end{equation*}
  (note that $p^s$ divides $p^r n - a_1$), together with $p|\boldsymbol{a}$,
  that
  \begin{equation*}
    \binom{p^r n}{\boldsymbol{a}} / \binom{p^{r - 1} n}{\boldsymbol{a}/ p} \equiv
     1 \pmod{p^{s + 2 t - \varepsilon}},
  \end{equation*}
  where $\varepsilon = 0$ if $p \geq 5$ and $\varepsilon = 1$ if $p = 3$.
  The same argument applies with $\boldsymbol{b}, \boldsymbol{c}, \boldsymbol{d},
  \boldsymbol{e}$ in place of $\boldsymbol{a}$. Suppose that the value of the
  quantity $s + 2 t$ is smallest for $\boldsymbol{a}$ compared to the
  corresponding values for $\boldsymbol{b}, \boldsymbol{c}, \boldsymbol{d},
  \boldsymbol{e}$ (this is no loss of generality; otherwise, we can swap
  $\boldsymbol{a}$ for one of $\boldsymbol{b}, \boldsymbol{c}, \boldsymbol{d},
  \boldsymbol{e}$ in the remaining argument). We therefore have
  \begin{equation}
    \frac{B (\boldsymbol{k})}{B (\boldsymbol{k}/ p)} \equiv 1 \pmod{p^{s + 2 t - \varepsilon}} . \label{eq:binomial:B:1}
  \end{equation}
  If $t = r$, then the claim follows and we are done. Otherwise, $\nu_p (a_2)
  = t < r$. It follows from the constraint $a_i + b_i + c_i + d_i + 2 e_i = 2
  p^r n$ that at least one of $b_2, c_2, d_2, e_2$ is not divisible by $p^{t +
  1}$. Without loss of generality, suppose that $\nu_p (b_2) \leq t$.
  Note that we have the divisibility \eqref{eq:summand:F:div} with $i = 2$ and
  that, by \eqref{eq:binomial:0:rs}, the two binomial coefficients on the
  left-hand side of \eqref{eq:summand:F:div} are each divisible by $p^{r -
  t}$. This shows that $B (\boldsymbol{k})$ is divisible by $p^{2 r - 2 t}$.
  Combining this with \eqref{eq:binomial:B:1}, we conclude that
  \begin{equation*}
    B (\boldsymbol{k}) \equiv B (\boldsymbol{k}/ p) \quad \left({\operatorname{mod} p^{2
     r + s - \varepsilon}}  \right)
  \end{equation*}
  for all $\boldsymbol{k} \in U (p^r n)$ with $p|\boldsymbol{k}$. Since $s
  \geq 1$ and $\varepsilon \in \{ 0, 1 \}$, this is slightly stronger, in
  the case $p|\boldsymbol{k}$, than the claimed congruences
  \eqref{eq:summand:F:p2}.
\end{proof}

The same argument we used to prove Theorem~\ref{thm:supercongruences:p2:F}
applies similarly to the sporadic Ap\'ery-like sequence labeled $(\delta)$
by Almkvist--Zudilin \cite{az-de06}. This sequence has the binomial sum
representation
\begin{equation}
  A_{\delta} (n) = \sum_{k = 0}^n (- 1)^k 3^{n - 3 k} \binom{n}{3 k} \binom{n
  + k}{n} \frac{(3 k) !}{k!^3}, \label{eq:Ad}
\end{equation}
solves the three-term recurrence \eqref{eq:apery3:rec:abcd} with $(a, b, c, d)
= (7, 3, 81, 0)$, and is also known as the Almkvist--Zudilin numbers. Combined
with Theorem~\ref{thm:supercongruences:p2:F}, the following proves
Theorem~\ref{thm:supercongruences:p2}.

\begin{theorem}
  \label{thm:supercongruences:p2:d}For all primes $p \geq 3$ and all
  positive integers $n, r$,
  \begin{equation*}
    A_{\delta} (p^r n) \equiv A_{\delta} (p^{r - 1} n) \pmod{p^{2
     r}} .
  \end{equation*}
\end{theorem}

\begin{proof}
  By expressing $A_{\delta} (n)$ as the constant term of $\Lambda (x, y, z)^n$
  where
  \begin{equation*}
    \Lambda (x, y, z) = \frac{(x + y - 1) (x + z + 1) (y - x + z) (y - z +
     1)}{x y z},
  \end{equation*}
  Gorodetsky \cite{gorodetsky-ct} obtained the alternative representation
  \begin{equation*}
    A_{\delta} (n) = \sum_{(\boldsymbol{a}, \boldsymbol{b}, \boldsymbol{c},
     \boldsymbol{d}) \in U (n)} (- 1)^{a_2 + b_1 + d_3} \binom{n}{\boldsymbol{a}}
     \binom{n}{\boldsymbol{b}} \binom{n}{\boldsymbol{c}} \binom{n}{\boldsymbol{d}}
  \end{equation*}
  as a multiple binomial sum. Here, the sum is over the set
  \begin{equation*}\arraycolsep=0.2em
    U (n) = \left\{ (\boldsymbol{a}, \boldsymbol{b}, \boldsymbol{c}, \boldsymbol{d})
     \in \mathbb{Z}_{\geq 0}^{12} \, : \quad \begin{array}{rcc}
       a_1 + a_2 + a_3 & = & n\\
       b_1 + b_2 + b_3 & = & n\\
       c_1 + c_2 + c_3 & = & n\\
       d_1 + d_2 + d_3 & = & n
     \end{array}, \quad \begin{array}{rcc}
       b_1 + c_1 + d_1 & = & n\\
       a_1 + b_2 + d_2 & = & n\\
       a_2 + b_3 + c_2 & = & n
     \end{array} \right\} .
  \end{equation*}
  As in the proof of Theorem~\ref{thm:supercongruences:p2:F}, we consider
  \begin{equation*}
    B (\boldsymbol{k}) \assign (- 1)^{a_2 + b_1 + d_3} \binom{n}{\boldsymbol{a}}
     \binom{n}{\boldsymbol{b}} \binom{n}{\boldsymbol{c}} \binom{n}{\boldsymbol{d}}
  \end{equation*}
  for $\boldsymbol{k}= (\boldsymbol{a}, \boldsymbol{b}, \boldsymbol{c}, \boldsymbol{d})
  \in U (n)$. The same argument as for the congruences \eqref{eq:summand:F:p2}
  applies and allows us to show that, again, $B (\boldsymbol{k}) \equiv B
  (\boldsymbol{k}/ p)$ modulo $p^{2 r}$, with the understanding that $B
  (\boldsymbol{k}/ p) = 0$ if $p \nmid \boldsymbol{k}$. As a consequence, we once
  more conclude that
  \begin{eqnarray*}
    A_{\delta} (p^r n) & = & \sum_{\substack{
      \boldsymbol{k} \in U (p^r n)\\
      p|\boldsymbol{k}
    }} B (\boldsymbol{k}) + \sum_{\substack{
      \boldsymbol{k} \in U (p^r n)\\
      p \nmid \boldsymbol{k}
    }} B (\boldsymbol{k})\\
    & \equiv & \sum_{\substack{
      \boldsymbol{k} \in U (p^r n)\\
      p|\boldsymbol{k}
    }} B (\boldsymbol{k}/ p) = \sum_{\substack{
      \boldsymbol{k} \in U (p^{r - 1} n)
    }} B (\boldsymbol{k}) = A_{\delta} (p^{r - 1} n) \pmod{p^2},
  \end{eqnarray*}
  as claimed.
\end{proof}

\section{The formal derivative of recurrence sequences}\label{sec:A:prime}

Suppose that $c_0 (n), \ldots, c_r (n) \in \mathbb{Z} [n]$ are polynomials
with $c_0 (0) = 0$ or, equivalently, $c_0 (n) \in n\mathbb{Z} [n]$ and $c_0
(n) \neq 0$ for all $n \in \mathbb{Z}_{> 0}$. Then there exists a unique
sequence $A (n)$ which satisfies the linear recurrence recurrence
\begin{equation}
  \sum_{j = 0}^r c_j (n) A (n - j) = 0 \label{eq:apery:rec:general}
\end{equation}
for all $n \geq 0$, subject to the initial conditions $A (0) = 1$ and $A
(j) = 0$ for $j < 0$. In the sequel, our interest will be limited to the cases
where recursion \eqref{eq:apery:rec:general} is either
\eqref{eq:apery2:rec:abc} or \eqref{eq:apery3:rec:abcd}. In particular, for
our purposes we have $r = 2$. Note that, for \eqref{eq:apery2:rec:abc} or
\eqref{eq:apery3:rec:abcd} only the initial condition $A (0) = 1$ is
significant because in \eqref{eq:apery:rec:general} with $n = 1$ the term
involving $A (- 1)$ vanishes due to $c_2 (1) = 0$.

Suppose further that $c_0 (n) \in n^2 \mathbb{Z} [n]$ (which is satisfied for
the recurrences \eqref{eq:apery2:rec:abc} and \eqref{eq:apery3:rec:abcd}).
Then we can introduce the formal derivative $A' (n)$ of $A (n)$ as the unique
sequence satisfying
\begin{equation}
  \sum_{j = 0}^r c_j (n) A' (n - j) + \sum_{j = 0}^r c_j' (n) A (n - j) = 0
  \label{eq:apery:rec:general:D}
\end{equation}
(where $c_j' (n) = \frac{\md}{\md n} c_j (n)$ is the ordinary derivative
of the polynomial $c_j (n)$), subject to the initial conditions $A' (j) = 0$
for $j \leq 0$.

We note that, in some cases, the formal derivative $A' (n)$ can be obtained as
a usual derivative of a natural interpolation of $A (n)$. Namely, suppose that
the sequence $A (n)$ can be extended to a smooth function $A (n)$ that is
defined for all real $n$, or a suitable subset of the reals, in such a way
that \eqref{eq:apery:rec:general} holds for all such $n$. By differentiating
\eqref{eq:apery:rec:general}, it then follows that the usual derivative
$\frac{\md}{\md n} A (n)$ satisfies the recursion
\eqref{eq:apery:rec:general:D}. Therefore, provided that the initial
conditions line up as well, the usual derivative agrees with the formal
derivative defined above. This is illustrated in Example~\ref{eg:AD:D}. We
then indicate in Example~\ref{eg:A3:D} that the same conclusion can still be
drawn if, instead of \eqref{eq:apery:rec:general}, the extension of $A (n)$ to
real $n$ satisfies an inhomogeneous recurrence with an additional term that
(together with its derivative) vanishes when $n$ is an integer.

\begin{example}
  \label{eg:AD:D}The sporadic sequence $\boldsymbol{D}$ (which is connected to
  $\zeta (2)$ in the same way that the Ap\'ery numbers \eqref{eq:apery3} are
  connected to $\zeta (3)$) is given by
  \begin{equation}
    A_{\boldsymbol{D}} (n) = \sum_{k = 0}^n \binom{n}{k}^2 \binom{n + k}{k}
    \label{eq:AD}
  \end{equation}
  and solves \eqref{eq:apery2:rec:abc} with $(a, b, c) = (11, 3, - 1)$. The
  sequence $A_{\boldsymbol{D}}' (n)$ is therefore characterized by solving the
  recurrence
  \begin{eqnarray}
    (n + 1)^2 A_{\boldsymbol{D}}' (n + 1) & = & (11 n^2 + 11 n + 3)
    A_{\boldsymbol{D}}' (n) + n^2 A_{\boldsymbol{D}}' (n - 1) 
    \label{eq:AD:D:rec}\\
    &  & - 2 (n + 1) A_{\boldsymbol{D}} (n + 1) + 11 (2 n + 1) A_{\boldsymbol{D}}
    (n) + 2 n A_{\boldsymbol{D}} (n - 1) \nonumber
  \end{eqnarray}
  with $A_{\boldsymbol{D}}' (0) = 0$ (note that we don't actually need an
  additional initial value for $A_{\boldsymbol{D}}' (- 1)$). The resulting
  initial values for $A_{\boldsymbol{D}}' (n)$ for $n = 1, 2, \ldots$ are
  \begin{equation*}
    5, \quad \frac{75}{2}, \quad \frac{1855}{6}, \quad
     \frac{10875}{4}, \quad \frac{299387}{12}, \quad
     \frac{943397}{4}, \quad \frac{63801107}{28}, \quad
     \frac{1253432797}{56}, \quad \ldots
  \end{equation*}
  Indeed, one can show (for instance, using creative telescoping) that the
  formal derivative has the explicit formula
  \begin{equation}
    A_{\boldsymbol{D}}' (n) = 5 \sum_{k = 0}^n \binom{n}{k}^2 \binom{n + k}{k}
    (H_n - H_k) \label{eq:AD:D}
  \end{equation}
  involving harmonic sums. On the other hand, the series
  \begin{equation*}
    A_{\boldsymbol{D}} (x) = \sum_{k = 0}^{\infty} \binom{x}{k}^2 \binom{x +
     k}{k}
  \end{equation*}
  converges for complex $x$ with $\Re x > - 1$ and therefore defines an
  interpolation of the sporadic sequence $A_{\boldsymbol{D}} (n)$. As shown in
  \cite{os-zagierapery}, this interpolation satisfies the homogeneous
  functional equation
  \begin{equation*}
    (x + 1)^2 A_{\boldsymbol{D}} (x + 1) - (11 x^2 + 11 x + 3) A_{\boldsymbol{D}}
     (x) - x^2 A_{\boldsymbol{D}} (x - 1) = 0
  \end{equation*}
  for all complex $x$ with $\Re x > - 1$ (this is
  \eqref{eq:apery2:rec:abc} with $(a, b, c) = (11, 3, - 1)$). We can then
  differentiate this equation to obtain \eqref{eq:AD:D:rec}. By verifying that
  the derivative of $A_{\boldsymbol{D}} (x)$ vanishes for $x = 0$, we conclude
  that, for positive integers $n$, the values $A_{\boldsymbol{D}}' (n)$ in
  \eqref{eq:AD:D} agree with the values of the actual derivative of the
  interpolation $A_{\boldsymbol{D}} (x)$.
\end{example}

\begin{example}
  \label{eg:A3:D}In a similar way, Zagier \cite[Section~7]{zagier-de}
  interpolates the Ap\'ery numbers $A (n)$, defined by the binomial sum
  \eqref{eq:apery3}, using the series
  \begin{equation*}
    A (x) = \sum_{k = 0}^{\infty} \binom{x}{k}^2 \binom{x + k}{k}^2,
  \end{equation*}
  which is well-defined for all complex $x$. Somewhat surprisingly, Zagier
  showed that, unlike the previous example, the interpolation $A (x)$
  satisfies the functional equation (see \cite{os-zagierapery} for an
  algorithmic derivation using creative telescoping)
  \begin{equation*}
    (x + 1)^3 A (x + 1) - (2 x + 1) (17 x^2 + 17 x + 5) A (x) + x^3 A (x - 1)
     = \frac{8}{\pi^2} (2 x + 1) \sin^2 (\pi x)
  \end{equation*}
  which is an inhomogeneous version of the recurrence
  \eqref{eq:apery3:rec:abcd} with $(a, b, c, d) = (17, 5, 1, 0)$ satisfied the
  Ap\'ery numbers. However, note that the inhomogeneous term and its
  derivative vanish for any integer $x$. We can therefore conclude as in the
  previous example that, for positive integers $n$, the values $A' (n)$ of the
  formal derivative agree with the derivative values of Zagier's
  interpolation. As recorded by Gessel \cite{gessel-super}, these are given
  by
  \begin{equation*}
    A' (n) = 2 \sum_{k = 0}^n \binom{n}{k}^2 \binom{n + k}{k}^2 (H_{n + k} -
     H_{n - k}) .
  \end{equation*}
\end{example}

\begin{remark}
  Write $F (x) = \sum_{n \geq 0} A (n) x^n$. Then the recurrence
  \eqref{eq:apery:rec:general} translates into the differential equation
  \begin{equation}
    \sum_{j = 0}^r c_j (\theta) x^j F (x) = 0, \label{eq:apery:diff:general}
  \end{equation}
  where $\theta = x \frac{\md}{\md x}$ is the Euler operator. Note that
  $c_0 (n)$ is the indicial polynomial of this differential equation. Since we
  assumed that $c_0 (n) \in n^2 \mathbb{Z} [n]$ and $c_0 (n) \neq 0$ for all
  $n \in \mathbb{Z}_{> 0}$, it follows that, up to scaling, $F (x)$ is the
  unique power series solution of \eqref{eq:apery:diff:general}. Moreover, if
  we set $G (x) = \sum_{n \geq 0} A' (n) x^n$, then we claim that $\log
  (x) F (x) + G (x)$ is a second solution of \eqref{eq:apery:diff:general}. To
  see this, note that it follows inductively from $\theta \log (x) F (x) =
  \log (x) \theta F (x) + F (x)$ that $\theta^k \log (x) F (x) = \log (x)
  \theta^k F (x) + k \theta^{k - 1} F (x)$. Consequently, we have
  \begin{equation*}
    c (\theta) \log (x) F (x) = \log (x) c (\theta) F (x) + c' (\theta) F (x)
  \end{equation*}
  for any polynomial $c (n) \in \mathbb{C} [n]$. On the other hand, we readily
  verify that $c (\theta) x^j = x^j c (\theta + j)$. Applying these formulae,
  as well as using \eqref{eq:apery:diff:general}, we find
  \begin{eqnarray*}
    \sum_{j = 0}^r c_j (\theta) x^j \log (x) F (x) & = & \sum_{j = 0}^r x^j
    c_j (\theta + j) \log (x) F (x)\\
    & = & \sum_{j = 0}^r x^j (\log (x) c_j (\theta + j) F (x) + c_j' (\theta
    + j) F (x))\\
    & = & \sum_{j = 0}^r x^j c_j' (\theta + j) F (x)\\
    & = & \sum_{j = 0}^r c_j' (\theta) x^j F (x) .
  \end{eqnarray*}
  Hence,
  \begin{equation*}
    \sum_{j = 0}^r c_j (\theta) x^j (\log (x) F (x) + G (x)) = \sum_{j = 0}^r
     [c_j (\theta) x^j G (x) + c_j' (\theta) x^j F (x)] = 0,
  \end{equation*}
  where the final equality follows from the equivalent recursion
  \eqref{eq:apery:rec:general:D} for the coefficients.
\end{remark}

\section{Gessel--Lucas congruences modulo \texorpdfstring{$p^2$}{p2}}\label{sec:proof}

We are now in a convenient position to prove Theorem~\ref{thm:lucas:p2:i},
which is restated below for convenience. This an extension of Gessel's result
\cite[Theorem~4]{gessel-super} for the Ap\'ery numbers and our proof
proceeds along the same lines, with an extra argument required for the three
sporadic Ap\'ery-like sequences that satisfy \eqref{eq:apery3:rec:abcd} with
$d \neq 0$. Here, the formal derivative $A' (n)$ is as introduced in
Section~\ref{sec:A:prime}.

\begin{theorem}
  Let $A (n)$ be one of the $6 + 6 + 3$ known sporadic Ap\'ery-like
  sequences. Then, for all primes $p \geq 3$ and all integers $n, k$ with
  $0 \leq k < n$,
  \begin{equation}
    A (p n + k) \equiv A (k) A (n) + p n A' (k) A (n) \pmod{p^2} .
    \label{eq:apery:lucas:p2}
  \end{equation}
\end{theorem}

\begin{proof}
  By Theorem~\ref{thm:lucas}, each sporadic Ap\'ery-like sequence $A (n)$
  satisfies the Lucas congruences. That is, $A (p n + k) \equiv A (k) A (n)$
  modulo $p$. We therefore have
  \begin{equation*}
    A (p n + k) \equiv A (k) A (n) + p a (n, k) \pmod{p^2}
  \end{equation*}
  for some $a (n, k) \in \mathbb{Z}_p$. Our goal is to show that $a (n, k)
  \equiv n A' (k) A (n) \pmod{p}$.
  
  By the case $r = 1$ of the supercongruences in
  Theorem~\ref{thm:supercongruences:p2}, we know that each of the sequences $A
  (n)$ in question satisfies the congruences
  \begin{equation*}
    A (p n) \equiv A (n) \pmod{p^2} .
  \end{equation*}
  This verifies the case $k = 0$ in \eqref{eq:apery:lucas:p2} because, by our
  definition, $A' (0) = 0$. In the sequel, we therefore assume that $k
  \geq 1$.
  
  Substituting $n$ by $p n + k$ in the recurrence \eqref{eq:apery:rec:general}
  for $A (n)$, which is either of the form \eqref{eq:apery2:rec:abc} or of the
  form \eqref{eq:apery3:rec:abcd}, we have
  \begin{equation*}
    \sum_{j = 0}^r c_j (p n + k) A (p n + k - j) = 0
  \end{equation*}
  for all $n, k \in \mathbb{Z}_{\geq 0}$. If we normalize $c_0 (n) =
  n^{\lambda}$ with $\lambda \in \{ 2, 3 \}$, then $c_2 (n) = c (n - 1)^2$ or
  $c_2 (n) = (n - 1) (c (n - 1)^2 + d)$ depending on whether
  \eqref{eq:apery:rec:general} is of the form \eqref{eq:apery2:rec:abc} or of
  the form \eqref{eq:apery3:rec:abcd}.
  
  We claim that the terms with $k - j < 0$ do not contribute modulo $p^2$.
  Note that $k - j < 0$ only if $k = 1, j = 2$, so that we need to show that
  $c_j (p n + k) A (p n + k - j) = c_2 (p n + 1) A (p n - 1)$ is divisible by
  $p^2$. We note that $c_2 (p n + 1)$ is divisible by $p^2$ unless we have
  $c_2 (n) = (n - 1) (c (n - 1)^2 + d)$ with $d \neq 0$. In that latter case,
  which consists of the $3$ known sporadic Ap\'ery-like sequences that
  satisfy \eqref{eq:apery3:rec:abcd} with $d \neq 0$, $c_2 (p n + 1)$ is only
  divisible by $p$. However, in those cases, we can combine
  Lemma~\ref{lem:A:pm1}, by which $A (p - 1)$ is divisible by $p$, with the
  Lucas congruences \eqref{eq:lucas:p:i} to conclude that
  \begin{equation*}
    A (p n - 1) = A (p (n - 1) + p - 1) \equiv A (n - 1) A (p - 1) \equiv 0
     \pmod{p} .
  \end{equation*}
  This shows that, in all cases, $c_2 (p n + 1) A (p n - 1)$ is divisible by
  $p^2$.
  
  Since the terms with $k - j < 0$ do not contribute modulo $p^2$ (and because
  $A (j) = 0$ for $j < 0$), we can apply the Lucas congruences to obtain
  \begin{equation*}
    \sum_{j = 0}^r c_j (p n + k) (A (k - j) A (n) + p a (n, k - j)) \equiv 0
     \pmod{p^2},
  \end{equation*}
  with the understanding that $a (n, j) = 0$ if $j < 0$. Using the Taylor
  expansion of the polynomials $c_j (n)$, this becomes
  \begin{equation*}
    \sum_{j = 0}^r (c_j (k) + p n c_j' (k)) (A (k - j) A (n) + p a (n, k -
     j)) \equiv 0 \pmod{p^2} .
  \end{equation*}
  Expanding, followed by applying the recurrence \eqref{eq:apery:rec:general},
  this is equivalent to
  \begin{equation}
    \sum_{j = 0}^r (c_j (k) a (n, k - j) + n c_j' (k) A (k - j) A (n)) \equiv
    0 \pmod{p} . \label{eq:apery:lucas:p2:rec:p}
  \end{equation}
  Since $c_0 (n) = n^{\lambda}$ with $\lambda \in \{ 2, 3 \}$, we have $c_0
  (k) \nequiv 0 \pmod{p}$ for all $k \in \{ 1, 2, \ldots, p
  - 1 \}$. Therefore, the congruence \eqref{eq:apery:lucas:p2:rec:p} together
  with $a (n, 0) = 0$ characterizes the values modulo $p$ of $a (n, k)$ for $k
  \in \{ 1, 2, \ldots, p - 1 \}$.
  
  On the other hand, replacing $n$ by $k$ in \eqref{eq:apery:rec:general:D},
  and multiplying with $n A (n)$, we find
  \begin{equation*}
    \sum_{j = 0}^r (c_j (k) n A' (k - j) A (n) + n c_j' (k) A (k - j) A (n))
     = 0.
  \end{equation*}
  Comparison with \eqref{eq:apery:lucas:p2:rec:p} shows that $a (n, k) \equiv
  n A' (k) A (n) \pmod{p}$, which is what we set out to
  show.
\end{proof}

\section{Conclusions}

Samol and van Straten \cite{sd-laurent09} (see also \cite{mv-laurent13})
showed that, if the Newton polytope of a Laurent polynomial $P (\boldsymbol{x})
\in \mathbb{Z} [\boldsymbol{x}^{\pm 1}]$, with $\boldsymbol{x}= (x_1, \ldots,
x_d)$, has the origin as its only interior integral point, then $A (n) =
\operatorname{ct} [P (\boldsymbol{x})^n]$, the sequence formed by the constant terms of
powers of $P (\boldsymbol{x})$, satisfies the \emph{Dwork congruences}
\begin{equation}
  A (p^r m + n) A (\lfloor n / p \rfloor) \equiv A (p^{r - 1} m + \lfloor n /
  p \rfloor) A (n) \pmod{p^r} \label{eq:dwork}
\end{equation}
for all primes $p$ and all integers $m, n \geq 0$, $r \geq 1$. The
case $r = 1$ of these congruences is equivalent to the Lucas congruences
\eqref{eq:lucas:i}. In a similar spirit, is there a natural extension of the
Gessel--Lucas congruences that we prove in Theorem~\ref{thm:lucas:p2:i} to
modulus $p^{2 r}$? Presumably, such an extension should also contain the
supercongruences of Theorem~\ref{thm:supercongruences:p2} as a special case.

In another direction, it is natural to investigate an extension of
Theorem~\ref{thm:lucas:p2:i} modulo $p^3$. More generally, it would be of
interest to understand the minimal number of states required for a linear
$p$-scheme describing sporadic Ap\'ery-like sequences modulo $p^r$. We refer
to recent work of Beukers \cite{beukers-p-schemes} for recent promising
progress in this regard.

Recently, through a careful and clever search, Gorodetsky was able to find new
constant term representations \cite{gorodetsky-ct} for several of the
sporadic Ap\'ery-like sequences. Combined with the result of Samol and van
Straten, these can be used to establish Theorem~\ref{thm:lucas} for $14$ of
the $15$ sporadic sequences. In the case of the sporadic sequence $(\eta)$,
however, we presently only have the lengthy and technical proof given in
\cite{ms-lucascongruences}. It would of interest to also find a suitable
constant term representation for the sporadic sequence $(\eta)$. More
generally, it would be valuable to have general results in the spirit of
\cite{sd-laurent09} that would allow us to prove
Theorem~\ref{thm:lucas:p2:i} in the presence of suitable constant term
expressions. However, while the result of Samol and van Straten
\cite{sd-laurent09} shows that large families of sequences satisfy the Lucas
congruences, the congruences in Theorem~\ref{thm:lucas:p2:i} are considerably
more rare. In particular, we note that the congruences \eqref{eq:lucas:p2:i}
imply the congruences $A (p n) \equiv A (n) \pmod{p^2}$,
which are an instance of the supercongruences discussed in
Section~\ref{sec:supercongruences}.

\begin{acknowledgements}
The author is grateful to Frits Beukers for
interesting discussions, motivation, and for sharing early versions of the
paper \cite{beukers-p-schemes}.
\end{acknowledgements}


\begin{thebibliography}{ABD19}
  \bibitem[ABD19]{abd-lucas}Boris Adamczewski, Jason~P. Bell, and {\'E}ric
  Delaygue. {\newblock}Algebraic independence of $G$-functions and congruences
  ``{\`a} la Lucas''. {\newblock}\textit{Annales Scientifiques de
  l'{\'E}cole Normale Sup\'erieure}, 52(3):515--559, 2019.
  
  \bibitem[AO00]{ahlgren-ono-apery}Scott Ahlgren and Ken Ono. {\newblock}A
  Gaussian hypergeometric series evaluation and Ap\'ery number congruences.
  {\newblock}\textit{Journal f{\"u}r die reine und angewandte Mathematik},
  2000(518):187--212, January 2000.
  
  \bibitem[Ap\'e79]{apery}Roger Ap\'ery. {\newblock}Irrationalit\'e de
  $\zeta (2)$ et $\zeta (3)$. {\newblock}\textit{Ast\'erisque}, 61:11--13,
  1979.
  
  \bibitem[AT16]{at-az}Tewodros Amdeberhan and Roberto Tauraso.
  {\newblock}Supercongruences for the Almkvist--Zudilin numbers.
  {\newblock}\textit{Acta Arithmetica}, 173:255--268, 2016.
  
  \bibitem[AZ06]{az-de06}Gert Almkvist and Wadim Zudilin.
  {\newblock}Differential equations, mirror maps and zeta values.
  {\newblock}In \textit{Mirror symmetry. V}, volume~38 of \textit{AMS/IP
  Stud. Adv. Math.}, pages 481--515. Amer. Math. Soc., Providence, RI, 2006.
  
  \bibitem[Beu85]{beukers-apery85}Frits Beukers. {\newblock}Some congruences
  for the Ap\'ery numbers. {\newblock}\textit{Journal of Number Theory},
  21(2):141--155, October 1985.
  
  \bibitem[Beu87]{beukers-apery87}Frits Beukers. {\newblock}Another congruence
  for the Ap\'ery numbers. {\newblock}\textit{Journal of Number Theory},
  25(2):201--210, February 1987.
  
  \bibitem[Beu22]{beukers-p-schemes}Frits Beukers. {\newblock}$p$-linear
  schemes for sequences modulo $p^r$. {\newblock}\textit{Preprint}, November
  2022. {\newblock}\href{http://arxiv.org/abs/2211.15240}{arXiv:2211.15240}.
  
  \bibitem[CCS10]{ccs-apery}Heng~Huat Chan, Shaun Cooper, and Francesco Sica.
  {\newblock}Congruences satisfied by Ap\'ery-like numbers.
  {\newblock}\textit{International Journal of Number Theory}, 6(1):89--97,
  2010.
  
  \bibitem[Coo12]{cooper-sporadic}Shaun Cooper. {\newblock}Sporadic sequences,
  modular forms and new series for $1 / \pi$. {\newblock}\textit{The
  Ramanujan Journal}, 29(1--3):163--183, 2012.
  
  \bibitem[Cos88]{coster-sc}Matthijs~J. Coster.
  {\newblock}\textit{Supercongruences}. {\newblock}PhD thesis, Universiteit
  Leiden, 1988.
  
  \bibitem[Del18]{delaygue-apery}Eric Delaygue. {\newblock}Arithmetic
  properties of Ap\'ery-like numbers. {\newblock}\textit{Compositio
  Mathematica}, 154(2):249--274, February 2018.
  
  \bibitem[Ges82]{gessel-super}Ira~M. Gessel. {\newblock}Some congruences for
  Ap\'ery numbers. {\newblock}\textit{Journal of Number Theory},
  14(3):362--368, June 1982.
  
  \bibitem[Ges83]{gessel-euler}Ira~M. Gessel. {\newblock}Some congruences for
  generalized Euler numbers. {\newblock}\textit{Canad. J. Math},
  35(4):687--709, 1983.
  
  \bibitem[Gor19]{gorodetsky-cong-q}Ofir Gorodetsky.
  {\newblock}$q$-congruences, with applications to supercongruences and the
  cyclic sieving phenomenon. {\newblock}\textit{International Journal of
  Number Theory}, 2019.
  
  \bibitem[Gor21]{gorodetsky-ct}Ofir Gorodetsky. {\newblock}New
  representations for all sporadic Ap\'ery-like sequences, with applications
  to congruences. {\newblock}\textit{Experimental Mathematics}, 2021.
  {\newblock}
  \href{https://doi.org/10.1080/10586458.2021.1982080}{DOI:10.1080/10586458.2021.1982080}.
  
  \bibitem[Gra97]{granville-bin97}Andrew Granville. {\newblock}Arithmetic
  properties of binomial coefficients I: Binomial coefficients modulo prime
  powers. {\newblock}\textit{CMS Conference Proceedings}, 20:253--275, 1997.
  
  \bibitem[HS22]{hs-lucas-x}Joel~A. Henningsen and Armin Straub.
  {\newblock}Generalized Lucas congruences and linear $p$-schemes.
  {\newblock}\textit{Advances in Applied Mathematics}, 141:1--20, \#102409,
  2022.
  
  \bibitem[Luc78]{lucas78}Edouard Lucas. {\newblock}Sur les congruences des
  nombres Eul\'eriens et des coefficients diff\'erentiels des fonctions
  trigonom\'etriques, suivant un module premier.
  {\newblock}\textit{Bulletin de la Soci\'et\'e Math\'ematique de
  France}, 6:49--54, 1878.
  
  \bibitem[McI92]{mcintosh-lucas}Richard~J. McIntosh. {\newblock}A
  generalization of a congruential property of Lucas.
  {\newblock}\textit{American Mathematical Monthly}, 99(3):231--238, March
  1992.
  
  \bibitem[Me4]{mestrovic-lucas}Romeo Me{\v s}trovi{\'c}. {\newblock}Lucas'
  theorem: its generalizations, extensions and applications (1878--2014).
  {\newblock}\textit{Preprint}, September 2014.
  {\newblock}\href{http://arxiv.org/abs/1409.3820}{arXiv:1409.3820}.
  
  \bibitem[MS16]{ms-lucascongruences}Amita Malik and Armin Straub.
  {\newblock}Divisibility properties of sporadic Ap\'ery-like numbers.
  {\newblock}\textit{Research in Number Theory}, 2(1):\#5, 26 p., 2016.
  
  \bibitem[MV16]{mv-laurent13}Anton Mellit and Masha Vlasenko.
  {\newblock}Dwork's congruences for the constant terms of powers of a Laurent
  polynomial. {\newblock}\textit{International Journal of Number Theory},
  12(2):313--321, 2016.
  
  \bibitem[OS11]{os-supercong11}Robert Osburn and Brundaban Sahu.
  {\newblock}Supercongruences for Ap\'ery-like numbers.
  {\newblock}\textit{Advances in Applied Mathematics}, 47(3):631--638,
  September 2011.
  
  \bibitem[OS13]{os-domb}Robert Osburn and Brundaban Sahu. {\newblock}A
  supercongruence for generalized Domb numbers.
  {\newblock}\textit{Functiones et Approximatio Commentarii Mathematici},
  48(1):29--36, March 2013.
  
  \bibitem[OS19]{os-zagierapery}Robert Osburn and Armin Straub.
  {\newblock}Interpolated sequences and critical $L$-values of modular forms.
  {\newblock}In J.~Bl{\"u}mlein, P.~Paule, and C.~Schneider, editors,
  \textit{Elliptic Integrals, Elliptic Functions and Modular Forms in
  Quantum Field Theory}, pages 327--349. Springer, 2019.
  
  \bibitem[OSS16]{oss-sporadic}Robert Osburn, Brundaban Sahu, and Armin
  Straub. {\newblock}Supercongruences for sporadic sequences.
  {\newblock}\textit{Proceedings of the Edinburgh Mathematical Society},
  59(2):503--518, 2016.
  
  \bibitem[Poo79]{alf}Alfred van~der Poorten. {\newblock}A proof that Euler
  missed ... Ap\'ery's proof of the irrationality of $\zeta (3)$.
  {\newblock}\textit{Mathematical Intelligencer}, 1(4):195--203, 1979.
  
  \bibitem[RY15]{ry-diag13}Eric Rowland and Reem Yassawi. {\newblock}Automatic
  congruences for diagonals of rational functions.
  {\newblock}\textit{Journal de Th\'eorie des Nombres de Bordeaux},
  27(1):245--288, 2015.
  
  \bibitem[RZ14]{rz-cong}Eric Rowland and Doron Zeilberger. {\newblock}A case
  study in meta-automation: automatic generation of congruence automata for
  combinatorial sequences. {\newblock}\textit{Journal of Difference
  Equations and Applications}, 20(7):973--988, 2014.
  
  \bibitem[SB85]{sb-picardfuchs}Jan Stienstra and Frits Beukers. {\newblock}On
  the Picard-Fuchs equation and the formal Brauer group of certain elliptic
  K3-surfaces. {\newblock}\textit{Mathematische Annalen}, 271(2):269--304,
  April 1985.
  
  \bibitem[Str14]{s-apery}Armin Straub. {\newblock}Multivariate Ap\'ery
  numbers and supercongruences of rational functions.
  {\newblock}\textit{Algebra \& Number Theory}, 8(8):1985--2008, 2014.
  
  \bibitem[Str22]{s-schemes}Armin Straub. {\newblock}On congruence schemes for
  constant terms and their applications. {\newblock}\textit{Research in
  Number Theory}, 8(3):1--21, \#42, 2022.
  
  \bibitem[SvS15]{sd-laurent09}Kira Samol and Duco van Straten.
  {\newblock}Dwork congruences and reflexive polytopes.
  {\newblock}\textit{Annales math\'ematiques du Qu\'ebec},
  39(2):185--203, October 2015.
  
  \bibitem[Zag09]{zagier4}Don~B. Zagier. {\newblock}Integral solutions of
  Ap\'ery-like recurrence equations. {\newblock}In \textit{Groups and
  symmetries}, volume~47 of \textit{CRM Proc. Lecture Notes}, pages
  349--366. Amer. Math. Soc., Providence, RI, 2009.
  
  \bibitem[Zag18]{zagier-de}Don~B. Zagier. {\newblock}The arithmetic and
  topology of differential equations. {\newblock}In Volker Mehrmann and Martin
  Skutella, editors, \textit{Proceedings of the European Congress of
  Mathematics, Berlin, 18-22 July, 2016}, pages 717--776. European
  Mathematical Society, 2018.
\end{thebibliography}
\end{document}